\documentclass{amsart}
\usepackage{amsmath,amssymb,amsthm,amsfonts}
\usepackage{mathrsfs,eufrak,euscript}
\usepackage{enumerate}
\theoremstyle{plain}
\newtheorem{thm}[subsection]{Theorem}
\newtheorem{prop}[subsection]{Proposition}
\newtheorem{lemma}[subsection]{Lemma}
\newtheorem{eg}[subsection]{Example}

\newtheorem{defn}[subsection]{Definition}
\newtheorem{rmk}[subsection]{Remark}

\newtheorem{note}[subsection]{Note}
\newtheorem{cor}[subsection]{Corollary}
\title[uniqueness of polar decomposition]{On the polar decomposition of right Linear operators in Quaternionic Hilbert spaces}
\author{G. Ramesh}
\address{G. Ramesh, Department of Mathematics, IIT Hyderabad, Kandi, Sangareddy(M), Medak (Dist), Telangana, India 502285.}
\email{rameshg@iith.ac.in}
\author{P. Santhosh Kumar}
\address{ P. Santhosh Kumar, Department of Mathematics, IIT Hyderabad, Kandi, Sangareddy(M), Medak (Dist), Telangana, India 502285.}
\email{ma12p1004@iith.ac.in}
\date{\today}

\begin{document}
\maketitle
\begin{abstract}
In this article we prove the existence of the polar decomposition for densely defined closed right linear operators in  quaternionic Hilbert spaces:
If $T$ is a densely defined closed right linear operator in a quaternionic Hilbert space $H$, then there exists a partial isometry $U_{0}$ such that $T = U_{0}|T|$.
In fact $U_{0}$ is unique if $N(U_{0}) = N(T)$. In particular, if $H$ is separable and $U$ is a partial isometry with $T = U|T|$, then we prove that $U = U_{0}$ if and only if either $N(T) = \{0\}$ or $R(T)^{\bot} = \{0\}$.
\end{abstract}

\section{Introduction}

The polar decomposition of an operator is a generalization of the polar decomposition of a non zero complex number.
If $T$ is a bounded linear operator between complex Hilbert spaces $H$ and $K$, then

\begin{equation}\label{polarform}
T=U_0|T|
\end{equation}
for some partial isometry $U_0$ with the  initial space $N(T)^{\bot}$, the orthogonal complement of the null space of $T$, a subspace of $H$ and   the final space $\overline{R(T)}$, the closure of the range of $T$, which is a subspace of $K$. Here $|T|$ is the square root of $(T^{*}T)$. Further, the partial isometry is unique if the null spaces of $T$ and $U_0$ are same. This factorization with the unique partial isometry $U_0$ is called the \text{polar decomposition} of $T$. This decomposition enable us to know the properties of $T$ by knowing the properties of the positive operator  $|T|$. It is known that the polar decomposition of a bounded operator defined between two complex Hilbert spaces always exist \cite[page 262]{Halmos}. The case of unbounded densely defined closed operator can be found in \cite[ page 218]{Pedersen}.

 In a recent article W. Ichinose and K. Iwashita gave a necessary and sufficient condition for any decomposition  as in Equation (\ref{polarform})  of bounded linear operator between complex Hilbert spaces to be the polar decomposition \cite{ichinose}.

 The polar decomposition of a bounded linear operator on a quaternionic Hilbert space is discussed in \cite[Theorem 2.8]{Ghiloni}. This result for the unbounded case is not found in the literature. In this article, we use quaternionic version of the bounded transform (see \cite{Alpay} for details ) to establish the existence of the polar decomposition of densely defined closed right linear operators in quaternionic Hilbert spaces.

  We also prove the result similar to that of Ichinose \cite[Theorem 1.2]{ichinose} for the case of densely defined closed right linear   operators defined in a separable quaternionic Hilbert space.  First we prove this for bounded operators by associating a $2\times 2$ matrix of complex operators for a given quaternionic operator. This technique for the case of quaternion matrices can be found in Brenner and Lee \cite{Brenner} and for bounded operators on $\ell^{2}$ in \cite{Feng}. Though every separable Hilbert space is isometrically isomorphic with $\ell^{2}$, the decomposition of the quaternion $\ell^{2}$ space into complex $\ell^{2}$ space is easy compared to the general space. To deal this general case we need to use the decomposition of the quaternionic Hilbert space into slice complex Hilbert spaces. Once this setting is done, all the results can be proved easily. In doing so, we give proofs of some results of \cite{Feng} for the separable Hilbert space.

Later, we prove the result for densely defined closed right linear operators defined in a separable quaternionic Hilbert space.

 We organize this article in five sections. In the second section we fix some of the notations, recall some basic properties of the ring of quaternions, definitions and some results in quaternionic Hilbert spaces.

 In the third section we prove the existence of square root of unbounded positive right linear operators in quaternionic Hilbert spaces.

 In the fourth section we prove the existence of the polar decomposition of densely defined closed right linear operators in quaternionic Hilbert spaces.

 In the fifth section we give necessary and sufficient condition for a decomposition as in Equation (\ref{polarform}), to be the polar decomposition. In this case, we assume the operator is defined in a separable quaternionic Hilbert space. 
\section{Notations and Prelimanaries}

Let $\mathbb{H}$ denotes the ring of all real quaternions. If $q \in \mathbb{H}$, then $q = q_{0}+q_{1}i+q_{2}j+q_{3}k$, where $q_{\ell} \in \mathbb{R}$, for each $\ell = 0, 1, 2 , 3$. Here $i,j,k$ satisfy
\begin{equation*}
i^{2}= j^{2}=k^{2} =  i\cdot j \cdot k = -1.
\end{equation*}
For $q\in \mathbb{H}$, the conjugate of $q$ is $\overline{q}= q_{0}-q_{1}i-q_{2}j-q_{3}k$. For $p,q \in \mathbb{H}$, we have $\overline{p\cdot q} = \overline{q}\cdot \overline{p}$ and $|q|: = \sqrt{q_{0}^{2}+q_{1}^{2}+q_{2}^{2}+q_{3}^{2}}$. The real part of $\mathbb{H}$ is denoted by Re($\mathbb{H}) = \left\{ q \in \mathbb{H} : q = \overline{q}\right\}$ and the imaginary part of $\mathbb{H}$ is denoted by Im($\mathbb{H}) = \left\{ q \in \mathbb{H} : q = -\overline{q}\right\}.$ The set $ \mathbb{S}:= \left\{ q \in \text{Im}(\mathbb{H}): |q| = 1 \right\}$ is called the imaginary unit sphere. For $m\in \mathbb{S},\; \mathbb{C}_{m}:= \{\alpha + m \beta : \alpha, \beta \in \mathbb{R}\}$ is called a slice of $\mathbb{H}$. Here $\mathbb{C}_{m}$ is  isomorphic to  the complex field $\mathbb{C}$ by the mapping $\alpha + m \beta \to \alpha + i \beta$.

\begin{defn} \cite[Definition 2.3]{Ghiloni}
	Let $ H $ be a right $ \mathbb{H}- $ module. A map
	\begin{equation*}
	\left\langle \cdot | \cdot \right\rangle \colon H \times H \longrightarrow \mathbb{H}
	\end{equation*}
	satisfying:
	\begin{enumerate}
		\item If $u \in H$, then $\left\langle u | u \right\rangle = 0 \Leftrightarrow u = 0$
		\item $\left\langle u | v +w \cdot q\right\rangle = \left\langle u | v\right\rangle  + \left\langle u | w \right\rangle \cdot q $, for all $u,v \in H$ and $q \in \mathbb{H}$
		\item $\left\langle u | v \right\rangle = \overline{\left\langle v | u \right\rangle }$, for all $u,v \in H$,
	\end{enumerate}
	is called an inner product on $H$. If we define  $\|u\|^2= \left\langle u | u\right\rangle$, for all $u \in H$, then $\|\cdot \|$ is a norm on $H$ and it is called the norm induced by $\left\langle \cdot | \cdot \right\rangle$. If $(H, \|\cdot \|)$ is complete space then it is called a right quaternionic Hilbert space.
\end{defn}
\begin{defn}
	Let $H$ be a right quaternionic Hilbert space. If $H$ has a countable dense subset then $H$ is called separable.
\end{defn}

\begin{defn}\cite{Alpay}
	A map $T \colon \mathcal{D}(T) \to H $ with the domain $\mathcal{D}(T) \subseteq H $ is said to be a right $\mathbb{H}- $ linear operator if
	\begin{equation*}
	T(x\cdot q + y) = T(x) \cdot q + T(y),\; \text{for every}\;  x,y \in \mathcal{D}(T)\; \text{and} \; q \in \mathbb{H}.
	\end{equation*}
	The null space of $T$ is defined by
 \begin{equation*}
 N(T) = \{x \in \mathcal{D}(T) \colon T(x) = 0\},
  \end{equation*}
 and the range space of $T$ is defined by
  \begin{equation*}
  R(T) = \{Tx \colon x \in \mathcal{D}(T)\}.
  \end{equation*}
  \end{defn}
  \begin{defn}\cite[Definition 2.1]{Alpay}
  Let $T \colon \mathcal{D}(T) \to H$ be right $\mathbb{H}$- linear operator. The graph of $T$ is denoted by $G(T)$ and it is defined by
  \begin{equation*}
  G(T) = \{(x,Tx) \colon x \in \mathcal{D}(T)\}.
  \end{equation*}
  The operator  $T$ is called
  \begin{enumerate}
  	\item closed,  if $G(T)$ is closed in $H\times H$. Equivalently, if $(x_{n}) \subset \mathcal{D}(T)$ with $(x_{n})$ converges to $x \in H$ such that  $Tx_{n}$ converges to $y \in H$ then $x\in \mathcal{D}(T)$ and $Tx=y$.
  	\item densely defined, if $\overline{\mathcal{D}(T)} = H$. (Here $\overline{\mathcal{D}(T)}$ denotes the closure $\mathcal{D}(T)$).
  \end{enumerate}
   \end{defn}
\begin{defn}\cite[Definition 2.2]{Alpay}
	Let $ T \colon \mathcal{D}(T) \to H $ be densely defined right $\mathbb{H}$- linear operator. Then there exists a unique right $\mathbb{H}$- linear opearator denoted by $T^{*}$  with the domain
	\begin{equation*}
	\mathcal{D}(T^{*}) = \{x \in H \colon \; \text{there exists unique} \; z \in H \; \text{ with} \; \left\langle x | Ty \right\rangle = \left\langle z| y\right\rangle\}
	\end{equation*}
	such that   $ \left\langle x | Ty\right\rangle = \left\langle T^{*}x | y\right\rangle $, for every  $ y \in \mathcal{D}(T), x \in \mathcal{D}(T^{*})$. This operator $ T^{*} $ is called the adjoint of $T.$
\end{defn}
\begin{thm}\cite[Theorem 2.3]{Alpay}
Let $T \colon \mathcal{D}(T) \to H$ be densely defined closed operator. Then
\begin{enumerate}
	\item $T^{*}$ is closed
	\item $R(T)^{\bot} = N(T^{*})$ and $N(T)^{\bot} = \overline{R(T^{*})}$.
\end{enumerate}
\end{thm}
By the closed graph theorem, if $T \colon \mathcal{D}(T) \to H$ is closed with $\mathcal{D}(T) = H$, then $T$  must be bounded.

\begin{defn}
	Let $T \colon \mathcal{D}(T) \to H$ and $S \colon \mathcal{D}(S) \to H$ be right $\mathbb{H}$- linear operators. Then we say that $ S \subseteq T$ if $\mathcal{D}(S) \subseteq \mathcal{D}(T)$ and $Sx = Tx$, for all $x \in \mathcal{D}(S)$.
	
	Furthermore
\begin{enumerate}
		\item $ST$ is defined with the domain $ \mathcal{D}(ST) = \{x \in \mathcal{D}(T) : Tx \in \mathcal{D}(S)\} $
		such that $ST(x) = S(Tx)$, for all $x \in \mathcal{D}(ST)$.
		\item $S+T$ is defined with the domain $\mathcal{D}(S+T)= \mathcal{D}(S) \cap \mathcal{D}(T)$ such that $S+T(x) = Sx + Tx$, for all $x \in \mathcal{D}(S+T)$.
\end{enumerate}
\end{defn}
\begin{defn}\cite[Definition 2.12]{Ghiloni}
	Let $T \colon \mathcal{D}(T) \to H$ be densely defined right $\mathbb{H}$- linear. Then $T$ is said to be
	\begin{enumerate}
		\item self-adjoint if $ T = T^{*}$
		\item positive if $T=T^{*}$ and $\left\langle x | Tx\right\rangle \geq 0 $, for all $x \in \mathcal{D}(T)$
		\item anti self-adjiont if $T^{*}= -T$
		\item normal if $T$ closed and $ TT^{*} = T^{*}T $
		\item unitary if $\mathcal{D}(T) = H$ and $T^{*}T = TT^{*}= I$.
	\end{enumerate}
\end{defn}
\begin{rmk}
Let $T \colon \mathcal{D}(T) \to H$ be right $\mathbb{H}$- linear. Then $T$ is said to be strictly positive if  $\left\langle x | Tx\right\rangle \geq \lambda \left\langle x | x\right\rangle$, for some $\lambda > 0 $ and for all $x \in \mathcal{D}(T)$.
\end{rmk}

\begin{defn}
	Let $T \colon \mathcal{D}(T) \to H$ be right $\mathbb{H}$- linear. Then $T$ commutes with bounded right linear operator $B$ on $H$ if and only if $TB \subseteq BT$. That is $TBx = BTx$, for all $x \in \mathcal{D}(T)$.
\end{defn}
\begin{thm} \cite[Theorem 2.10]{Ghiloni}
	Let $T \colon H \to H$ be right $\mathbb{H}$- linear. Then $T$ is bounded  or continuous, if there exists $K>0$ such that
	\begin{equation*}
	\|Tx\| \leq K \|x\|, \; \text{for all} \; x \in H.
	\end{equation*}
	We denote the set of all bounded right $\mathbb{H}$- linear operators on $H$ by $\mathcal{B}(H)$. If $T\in \mathcal{B}(H)$, then
	\begin{equation*}
	\|T\| = \sup \left\{\|Tx\|: x\in H, \|x\|=1\right\}
	\end{equation*}
	is finite and it is called the norm of $ T$.
\end{thm}

\begin{defn}
	Let $T \in \mathcal{B}(H)$. Then $T$ is said to be
	\begin{enumerate}
		\item isometry if $\|Tx\| = \|x\|, \; \text{for all} \; x \in H$
		\item partial isometry if $\|Tx\| = \|x\|, \;\text{for all}\; x \in N(T)^{\bot}$. The space $N(T)^{\bot}$ called the initial space and the space $R(T)$ called the final space of $T$.
		\item orthogonal projection onto closed subspace $W$ if $T=T^{*} , \; T^{2}=T$ and $R(T) =W$. We denote this by $T = P_{W}$.
	\end{enumerate}
	
\end{defn}
We recall existence of square root of bounded positive operator on quaternionic Hilbert space, which is given in \cite[Theorem 2.18]{Ghiloni}.

\begin{thm}
	If $T \in \mathcal{B}(H)$ and $T$ is positive. Then there exists a unique operator in $\mathcal{B}(H)$, indicated by $T^{\frac{1}{2}}$, such that $T^{\frac{1}{2}}$ is positive and $T^{\frac{1}{2}} T^{\frac{1}{2}} = T$. Furthermore, it turns out that $T^{\frac{1}{2}}$ commutes with every bounded operator which commutes with $T$.
\end{thm}
Where as, we prove the existence of square roof of unbounded positive opeartor in quaternionic Hilbert spaces (for details, see Theorem \ref{strictlysqareroot} of section 3).
\begin{rmk}\cite[Remark 2.13(1)]{Ghiloni}
If $J \colon H \to H$ be anti self-adjoint and unitary then $J \in \mathcal{B}(H)$ and $\|J\| = 1$.
\end{rmk}
\subsection{Slice Hilbert space:}

Given any quaternionic Hilbert space $H$ and an anti self-adjoint unitary operator $J$ on $H$  there exists a complex Hilbert space called the slice Hilbert space, defined as follows:
\begin{defn}\cite[Definition 3.6]{Ghiloni}\label{slice}
	Let $m \in \mathbb{S}$ and $J \in \mathcal{B}(H)$ be an anti self-adjoint and unitary operator. Then
	\begin{equation*}
	H^{Jm}_{\pm}:= \{x\in H \colon Jx = \pm \ x \cdot m\}
	\end{equation*}
	is called a slice Hilbert space.
\end{defn}
\begin{lemma}\cite[Lemma 3.10]{Ghiloni}\label{directsum}
	Let $H$ and $J$ be as in Definition \ref{slice}. Then
	\begin{enumerate}
	\item $H^{Jm}_{\pm} \neq \{0\}$ and the restriction of the inner product $\left\langle \cdot | \cdot \right\rangle_{\mathbb{H}}$ to $H^{Jm}_{\pm}$ is complex valued. Therefore $H^{Jm}_{\pm}$ is $\mathbb{C}_{m}$ - Hilbert space.
	\item $H = H^{Jm}_{+} \oplus H^{Jm}_{-}$.
	\end{enumerate}
\end{lemma}
\begin{prop}\cite[Proposition 3.11]{Ghiloni}\label{antilineariso}
Let $m,n \in \mathbb{S}$ with $m \cdot n = -n \cdot m$. The map $\Phi \colon H \to H$ defined by $\Phi(x) = x \cdot n$, is a complex anti-linear isomorphism and $\Phi(H^{Jm}_{+}) = H^{Jm}_{-}$.
\end{prop}
\begin{note}
	\bf{Througout this article $m \in \mathbb{S}$ is fixed} and $n \in \mathbb{S}$ such that $m \cdot n = - n \cdot m$.
\end{note}
\begin{lemma}\cite[Proposition 3.11]{Ghiloni}\label{innerdirect}
Let $x \in H^{Jm}_{+}$ and $y \in H^{Jm}_{-}$. Then
\begin{equation*}
\left\langle x | y\right\rangle_{\mathbb{H}} + \left\langle y | x \right\rangle_{\mathbb{H}} = 0.
\end{equation*}
\end{lemma}
\begin{eg}
	Let
	\begin{equation*}
	H= \ell^{2}(\mathbb{N}, \mathbb{H})= \big\{(q_{n})_{n \in \mathbb{N}}: \forall n \in \mathbb{N}, q_{n} \in \mathbb{H},\sum_{n\in \mathbb{N}}|q_{n}|^{2} < \infty\big\}.
	\end{equation*}
  The inner product is given by
\begin{equation*}
\left\langle x | y\right\rangle_{\mathbb{H}} = \sum_{n \in \mathbb{N}}\overline{x}_{n}y_{n}, \;\text{for all}\; x = (x_{n}), y = (y_{n}) \in H.
\end{equation*}
Then $H$ is separable with an orthonormal basis $\{e_{n}: n \in \mathbb{N}\}$, where $e_{n}(k) = \delta_{nk}$.
Let
\begin{equation*}
\ell^{2}(\mathbb{N}, \mathbb{C}) = \{(\alpha_{n})_{n \in \mathbb{N}}: \forall n \in \mathbb{N}, \sum_{n\in \mathbb{N}}|\alpha_{n}|^{2} < \infty\}.
\end{equation*}
 The inner product is given by,
\begin{equation*}
\left\langle \alpha | \beta \right\rangle_{\mathbb{C}} = \sum_{n \in \mathbb{N}} \overline{\alpha_{n}} {\beta}_{n}, \; \text{for all}\; \alpha = (\alpha_{n}), \beta = (\beta_{n}) \in \ell^{2}(\mathbb{N},\mathbb{C}).
\end{equation*}
It is a separable complex Hilbert space.

If  $q = q_{0}+q_{1}i+q_{2}j+q_{3}k\in \mathbb{H}$, then there exists $\alpha, \beta \in \mathbb{C}$ such that $q = \alpha + \beta \cdot j$. In fact $\alpha = q_{0}+q_{1} i$ and $\beta = q_{2}+q_{3}i$. So for every $(q_{n}) \in \ell^{2}(\mathbb{N}, \mathbb{H})$, we have two complex sequences $(\alpha_{n}), (\beta_{n})$ such that
\begin{equation*}
(q_{n}) = (\alpha_{n}) + (\beta_{n}) \cdot j
\end{equation*}
and
\begin{equation*}
\|(q_{n})\|^{2} = \|(\alpha_{n})\|^{2} + \|(\beta_{n})\|^{2}.
\end{equation*}
This shows that $(\alpha_{n}), (\beta_{n}) \in \ell^{2}(\mathbb{N}, \mathbb{C})$.  Define $J \colon H \to H$ by
\begin{equation*}
J(\alpha_{n} + \beta_{n} \cdot j) = (\alpha_{n}- \beta_{n}\cdot j) \cdot i, \; \text{for all} \; (q_{n}) \in H.
\end{equation*}
Then $J$ is anti self-adjoint, unitary and we can show that $H^{Ji}_{+} = \ell^{2}(\mathbb{N}, \mathbb{C})$.
\end{eg}

\section{Square root of an unbounded positve operator}
In this section we prove the existence of square root of an unbounded right linear positive operator.
\begin{thm}\label{strictlysqareroot}
	Let $T \colon \mathcal{D}(T) \to H $ be positive. Then there exists a  unique positive square root of $T$.  Moreover, it commutes with every bounded operator that commutes with $T$.
\end{thm}
\begin{proof} This proof is based on \cite[Theorem 1]{Wouk}. We devide this  proof into two cases. First we prove the result when $T$ is strictly positive, later we prove it for positive operator.
	
	Case$(1):$ Assume that $T$ is strictly positive. Then
	\begin{equation*}
	\left\langle x | Tx\right\rangle \geq \lambda \left\langle x | x\right\rangle,\; \text{for some}\; \lambda > 0 \; \text{and \ for all} \; x \in \mathcal{D}(T).
	\end{equation*}
	 By the Caucy-Schwarz inequality \cite[Proposition 2.2]{Ghiloni}, we have
	\begin{equation} \label{boundedbelow}
	\|Tx\| \geq \lambda \|x\|,\; \text{ for all} \; x \in \mathcal{D}(T).	\end{equation}
	Since $T$ is positive and bounded below by Equation $(\ref{boundedbelow})$, we conclude that $T^{-1}$ exists and it is bounded. Moreover, $T^{-1}$ is positive
	\begin{equation*}
	\left\langle y | T^{-1}y\right\rangle = \left\langle Tx | x\right\rangle \geq \lambda \|x\|^{2}\geq 0, \; \text{for all}\; y \in H.
	\end{equation*}

	  Let $S:= T^{-1}$.	By \cite[Theorem 2.18]{Ghiloni}, $S$ has  unique positive  square root $S^{\frac{1}{2}}$ with $\|S^{\frac{1}{2}}\| \leq \frac{1}{\sqrt{\lambda}}$. Moreover, $S^{\frac{1}{2}}$ commutes with every bounded operator that commutes with $S$. Since $S$ commutes with every bounded operator that commutes with $T$, so does $S^{\frac{1}{2}}$. Clearly $R(S) \subset R(S^{\frac{1}{2}})$ and $N(S^{\frac{1}{2}}) = N(S) = N(T^{-1}) = \{0\}$. Then  $S^{\frac{1}{2}}$ is one to one from $H$ onto $R(S^{\frac{1}{2}})$. Thus  $S^{\frac{1}{2}}$ has unbounded inverse say $C$ with domain  $\mathcal{D}(C) = R(S^{\frac{1}{2}})$. We show that $\mathcal{D}(C^{2}) = \mathcal{D}(T)$ and $C^{2}= T$. If $x \in \mathcal{D}(T)$ then $Tx = y$ implies $x= S^{\frac{1}{2}} S^{\frac{1}{2}} y$, hence $C^{2}x = y$. Therefore $\mathcal{D}(T) \subseteq \mathcal{D}(C^{2})$.
	
	  Conversely if $x \in \mathcal{D}(C^{2})$ then $y := Cx \in \mathcal{D}(C)$. That is $x = S^{\frac{1}{2}}y$ and $C^{2}x = Cy = z$. This implies $S^{\frac{1}{2}} y = x$ and  $S^{\frac{1}{2}}z = y$. Thus $z = Tx$. This shows that $\mathcal{D}(T) = \mathcal{D}(C^{2})$ and $C^{2}x = (S^{\frac{1}{2}})^{-2}x = S^{-1}x = Tx$, for all $x \in \mathcal{D}(T)$. In fact $C$ is positive
	  \begin{equation*}
	  \left\langle S^{\frac{1}{2}}x|CS^{\frac{1}{2}}x\right\rangle = \left\langle S^{\frac{1}{2}}x|x\right\rangle \geq 0, \; \text{for all}\;  S^{\frac{1}{2}}x\in R(S^{\frac{1}{2}}) = \mathcal{D}(C).
	  \end{equation*}
	
As $T$ is strictly positive, $T$ is invertible so is  $C$.  Since $C$ being positive and invertible, it must be strictly positive.

 Since $C$ commutes with every bounded operator that commutes with $S^{\frac{1}{2}}$ so $C$ commutes with every bounded operator that commutes with $T$. 	It is enough to show the uniqueness. If $L$ is any positive square root of $T$ then for all $x \in \mathcal{D}(T)$
	\begin{equation*}
	\left\langle x | Tx\right\rangle = \left\langle x | L^{2}x\right\rangle = \left\langle Lx | Lx\right\rangle \geq \lambda \left\langle x | x \right\rangle .
	\end{equation*}
	So $L$ has a bounded inverse which is positive. Then $(L^{-1})^{2}Tx = x$, for all $x\in \mathcal{D}(T)$, where as $T(L^{-1})^{2}x = x,$ for all $x \in H$. Thus $S = (L^{-1})^{2}$. By the uniqueness of square root of $S$, we have $S^{\frac{1}{2}} = L^{-1}$. Therefore $L = C$.
	
	\noindent Case$(2):$ Assume that  $T$ is positive. Define
	\begin{equation*}
	S := T(I+T)^{-1}
	\end{equation*}
	We show that $S$ is a bounded positive operator. Since $T$ is positive $I+T$ is bounded below. That is
	\begin{equation*}
	\|x\|\|(I+T)x\| \geq \left\langle x | (I+T)x\right\rangle =  \left\langle x | x\right\rangle+\left\langle x | Tx\right\rangle \geq \|x\|^{2}, \;\text{for all} \; x \in \mathcal{D}(T) = \mathcal{D}(I+T).
	\end{equation*}
	Then $(I+T)^{-1}$ exists and it is bounded with $\|(I+T)^{-1}\| \leq 1$. The operator $S$ can be represented as
	\begin{equation*}
	S = T(I+T)^{-1} = (I+T-I)(I+T)^{-1} = I - (I+T)^{-1}.
	\end{equation*}
	By using above representation, we show that $S$ is bounded positive operator
	\begin{equation*}
	\|Sx\| \leq \|x\| + \|(I+T)^{-1}x\| \leq 2 \|x\|, \; \text{for all}\; x \in H.
	\end{equation*}
	
	By Case $(1)$, the operator $I+T$ has positive square root say $C$,  which commutes with every bounded operator that commutes with $I+T$. Hence $C$ commutes with every bounded operator that commutes with $(I+T)^{-1}$. Since $S$ commutes with $I+T$ so it commutes with both $(I+T)^{-1}$ and $C$. It is clear that $\mathcal{D}(C^{2}) = \mathcal{D}(I+T) = \mathcal{D}(T)$. Thus for every $x\in \mathcal{D}(T)$, we have
	\begin{equation*}
	(S^{\frac{1}{2}} C)^{2}x = Tx.
	\end{equation*}
	
	Since $S^{\frac{1}{2}}C$ is positive operator, we conclude that  $(S^{\frac{1}{2}} C)$ is positive square root of $T$. Also it commutes with every bounded operator that commutes with $T$. The uniqueness of square root  follows same as in the Case $(1)$.
\end{proof}
\begin{cor}
Let $T \colon \mathcal{D}(T) \to H$ be right $\mathbb{H}$ - linear operator. Then there exists a unique posiitive square root of $T^{*}T$ called modulus of $T$ denoted by $|T|$. Moreover, $|T|$ commutes with every bounded operator that commutes with $T$.
\end{cor}
\begin{proof}
It is clear that $T^{*}T$ is positive operator, that is
\begin{equation*}
\left\langle x | T^{*}Tx\right\rangle = \left\langle Tx|Tx\right\rangle = \left\langle T^{*}Tx|x\right\rangle \geq 0, \; \text{for all}\; x \in \mathcal{D}(T^{*}T).
\end{equation*}
Hence the result follows from Theorem \ref{strictlysqareroot}.
\end{proof}
\section{Existence of polar decomposition}
In this section we prove the existence of polar decomposition of a densely defined closed right $\mathbb{H}$- linear operator in a quaternionic Hilbert space $($see Theorem \ref{Existence}$)$.

We recall the notion of the polar decomposition of a bounded linear operator on complex Hilbert spaces.

\begin{thm}\label{classical}\cite[Theorem 3.2.17]{Pedersen}
	If $H$ is complex Hilbert space and let $A \colon H \to H $ be bounded. Then there is a unique partial isometry  $U_{0} \colon H \to H $ with $N(U_{0}) = N(A)$ and  $A = U_{0}|A|$ . In particular, $U_{0}^{*}U_{0}|A| = |A|, U_{0}^{*}A = |A|$, and $U_{0}U_{0}^{*}A= A$.
\end{thm}
 Through out this article $U_{0}$ represents patial isometry as in Theorem \ref{classical}, the polar decomposition. Now we recall existence of polar decomposition of bounded right linear operator on quaternionic Hilbert spaces \cite{Ghiloni}.

\begin{thm}\cite[Theorem 2.20]{Ghiloni}\label{boundedexistence}
	Let $T \in \mathcal{B}(H)$. Then there exists, and are unique, two operators $W$ and $P$ in $\mathcal{B}(H)$ such that:
	\begin{enumerate}[(i)]
	\item $T= WP$
	\item $P \geq 0$
	\item $N(P) \subseteq N(W), \; ($ in fact $N(P) = N(W) )$
	\item $W$ is isometry on $N(P)^{\bot}$ ; that is, $\|Wx\| = \|x\|$, for every $x \in N(P)^{\bot}$.
\end{enumerate}
	\noindent Furthermore, $W$ and $P$ have the following additional properties:
	\begin{enumerate}[(a)]
	\item $P = |T|$
	\item If $T$ is normal, then $W$ defines a unitary operator in $\mathcal{B}(\overline{R(T)})$
	\item If $T$ is normal, then $W$ commutes with $|T|$ and with all the operators in $\mathcal{B}(H)$ commuting with both $T$ and $T^{*}$.
	\item If $T$ is self-adjoint, then $W$ is self-adjoint
	\item If $T$ is anti self-adjoint, then $W$ is anti self-adjoint. 	
	\end{enumerate}
	\end{thm}
A densely defined closed right $\mathbb{H}$- linear operator in a quaternionic Hilbert space $H$ can be transformed to a bounded linear operator on $H$ via the  bounded transform.

\begin{thm}\cite[Theorem 6.1]{Alpay} \label{boundedtransform}
	Let $T \colon \mathcal{D}(T) \to H$ be densely defined closed right linear operator in $H$. The operator $Z_{T} := T(I+T^{*}T)^{-\frac{1}{2}}$ has the following properties:
	\begin{enumerate}
	\item $Z_{T} \in \mathcal{B}(H), \; \|Z_{T}\| \leq 1 $ and $
	T = Z_{T} (I-Z_{T}^{*}Z_{T})^{-\frac{1}{2}}$
	
\item $Z_{T}^{*} = Z_{T^{*}}$
\item If $T$ is normal, then $Z_{T}$ is normal.	
	\end{enumerate}
\end{thm}
\begin{lemma}\label{null}
	If $T\colon \mathcal{D}(T)\to H$ be densely defined closed right $\mathbb{H}$- linear operator, then  $N(Z_{T})=N(T)$ and $R(Z_{T})=R(T).$
\end{lemma}
\begin{proof} By the definition of $Z_{T}$, we have $R(Z_{T}) \subseteq R(T)$.  Since $T = Z_{T}(I-Z_{T}^{*}Z_{T})^{-\frac{1}{2}}$, we have $R(T) \subseteq R(Z_{T})$. Similarly we have $R(T^{*}) = R(Z_{T}^{*})$. Hence $N(T) = N(Z_{T})$. 	
\end{proof}

\begin{thm}\label{Existence}
		Let $T \colon \mathcal{D}(T) \to H$ be densely defined closed right $\mathbb{H}$- linear operator with $\mathcal{D}(T) \subset H$. Then there exists a  partial isometry $U_{0}$ with initial space $N(T)^{\bot}$ and final space $\overline{R(T)}$ such that $T = U_{0}|T|$.  In particular $U_{0}$  satisfying $N(|T|) = N(U_{0})$ is uniquely determined.
Moreover,
		\begin{enumerate}
			\item \label{normalpolar} If $T$ is normal, then $U_{0}$ is unitary operator in $\mathcal{B}(\overline{R(T)})$
			\item \label{commutingpolarnormal} If $T$ is normal, then $U_{0}$ commutes with $|T|$ and with all operators in $\mathcal{B}(H)$ commuting with both $T$ and $T^{*}$
			\item \label{selfadjpolar} If $T$ is self-adjoint, then also $U_{0}$
			\item \label{antiselfpolar} If $T$ is anti self-adjoint, then also $U_{0}$.
		\end{enumerate}
	\end{thm}
	\begin{proof} It is clear from Theorem \ref{boundedtransform} that  $Z_{T}$ is bounded. By Theorem \ref{boundedexistence}, there exists a unique partial isometry $U_{0}$
		with initial space $N(Z_{T})^{\bot}$, final space $\overline{R(Z_{T})}$ such that $Z_{T}= U_{0}|Z_{T}|$ and $ N(|{Z}_{T}|) = N(U_{0})$. Since $|T| = |Z_{T}|(I-Z_{T}^{*}Z_{T})$, we have
		\begin{align*}
		T &= Z_{T}(I-Z_{T}^{*}Z_{T})^{-\frac{1}{2}} = U_{0} |Z_{T}|(I-Z_{T}^{*}Z_{T})^{-\frac{1}{2}} = U_{0} |T|.
		\end{align*}
		Hence  $T = U_{0}|T|$. By Lemma \ref{null} and Theorem \ref{boundedexistence}({\it{iii}}) it follows that $N(|T|) = N(U)$.

		 Proof of  (\ref{normalpolar}): If $T$ is normal then by Theorem \ref{boundedtransform}(3), we have $Z_{T}$ is normal. By Theorem \ref{boundedexistence}(b), $U_{0}$ is unitary operator in $\mathcal{B}(\overline{R(Z_{T})})= \mathcal{B}(\overline{R(T)})$.
		
		 Proof of (\ref{commutingpolarnormal}): Since $Z_{T}$ is normal, by Theorem \ref{boundedexistence}(c), $U_{0}$ commutes with $|Z_{T}|$ and with all operators in $\mathcal{B}(H)$ commuting with both $Z_{T}$ and $Z_{T}^{*}$. We conclude the result by using the fact that $|T|$ commutes with $|Z_{T}|$.
		
		 Proof of (\ref{selfadjpolar}): If $T$ is self-adjoint then by Theorem \ref{boundedtransform}(2), $Z_{T}$ is self-adjoint. By Theorem \ref{boundedexistence}({d}), we have $U_{0}$ is self-adjoint.

		Proof of (\ref{antiselfpolar}): If $T$ is anti self-adjoint, then
		\begin{align*}
		Z_{T}^{*} = Z_{T^{*}} = Z_{-T}
		                      = -T\left(I+(-T)^{*}(-T)\right)^{-\frac{1}{2}}= -T(I-T^{2})^{-\frac{1}{2}}= -Z_{T}.
		\end{align*}
		Hence by Theorem \ref{boundedexistence}(e), we get $U_{0}$ is anti self-adjoint.
	\end{proof}
	
\begin{cor}
Let $T \colon \mathcal{D}(T) \to H$ be right $\mathbb{H}$- linear normal operator which is densely defined and closed. Let $T = V|T|$ be the polar decomposition as in Theorem \ref{Existence}. Then
\begin{enumerate}
	\item $V$ is normal
	\item there exists a unitary $W$ such that $T=W|T|$.
\end{enumerate}
\end{cor}
\begin{proof} Since $T$ is normal operator, $N(T) = N(T^{*})$. We use this fact in proving the result.
	
\noindent	Proof of $(1):$ We know that $V$ is partial isometry. This implies
\begin{equation*}
V^{*}V = P_{N(V)^{\bot}}= P_{N(T)^{\bot}} = P_{N(T^{*})^{\bot}} = P_{N(V^{*})^{\bot}} = VV^{*}.
\end{equation*}
Proof of $(2):$
Clearly  $V \colon \overline{R(|T|)} \to R(T)$ is an isometry such that $T=V|T|$. If we define $W$ as
\begin{align*}
Wx &= Vx , \; \text{if} \; x \in \overline{R(|T|)}\\
   &= x, \; \text{if} \; x \in N(T),
\end{align*}
then $W$ is unitary. If  $x \in \mathcal{D}(T)$ then $x = x_{1} + x_{2}$, where $x_{1} \in N(T) $ and $x_{2} \in \overline{R(|T|)} $. Moreover,
\begin{equation*}
T(x_{1}+x_{2}) = V|T|(x_{1}+x_{2}) = V|T|(x_{2}) = W|T|(x_{2}) = W|T|(x_{1} + x_{2}).
\end{equation*}
Hence the result.
	
\end{proof}	
\section{ uniqueness of polar decomposition}

In this section we give necessary and sufficient condition for any decomposition of densely defined closed right $\mathbb{H}$- linear operator in quaternionic Hilbert space $H \; ($ see Theorem \ref{unboundedunique}$)$ as in Equation (\ref{polarform}) to be the polar decomposition.  Throughout this section we assume $H$ to be separable.

First, we prove the result for bounded right $\mathbb{H}$- linear operators on $H$ and then extend it to the  unbounded case.

In case of complex Hilbert spaces the uniqueness of decomposition of bounded linear operator is given in \cite[Theorem 1.2]{ichinose}. In detail, let $A$ be a bounded linear operator on a complex Hilbert space with
\begin{equation}\label{decomposition}
A = U|A|,
\end{equation}
 where $U$ is a partial isometry.  In this case $U$ may not be unique. For example let $P$ be an orthogonal projection onto $N(A)$, and $V \colon N(A) \to R(A)^{\bot}$ be a  partial isometry. Then $U=U_{0}+ VP$ also satisfy Equation (\ref{decomposition}). In fact, it is proved in \cite[Proposition 2.5]{ichinose} that any partial isometry $U$ satisfying Equation (\ref{decomposition}) above can be written as $U=U_0+VP$, where $U_0$ is a partial isometry satisfying Equation(\ref{decomposition}) such that $N(U_{0}) = N(A)$. Thus $U = U_{0}$ if and only if $N(A) = \{0\}$ or $R(A)^{\bot} = \{0\}$.

First we prove some results which are needed for our purpose.
\begin{lemma}
 Let $J \in \mathcal{B}(H)$ be anti self-adjoint and unitary. Suppose $\{f_{k}\}_{k=1}^{\infty}$ is an orthonormal basis in $H$ with $f_{k} = g_{k} + h_{k} \cdot n,$ where $\{g_{k}\}_{k=1}^{\infty}, \{h_{k}\}_{k=1}^{\infty} \subset H^{Jm}_{+}.$ Then $\{g_{k}\}$ is an orthonormal basis in $H^{Jm}_{+}$ if and only if $h_{k}=0,$ for $ k= 1,2,3,\cdots$. Similarly $\{h_{k}\}$ is an orthonormal basis in $H^{Jm}_{+}$ if and only if $g_{k}=0,$ for $k= 1,2,3,\cdots.$
\end{lemma}
\begin{proof}
	Proof is similar to that of  \cite[Lemma 2.4]{Feng}.
\end{proof}
\begin{lemma}
Let $\{f_{k}: k = 1,2,3, \cdots \}$ be an orthonormal basis in $H$ and $A \in \mathcal{B}(H).$ Then the matrix of $A$ with respect to $\{f_{k}: k = 1,2,3, \cdots \}$ is given by $A = \big( q_{rs}\big)$, where $q_{rs} = \left\langle f_{r}| Af_{s}\right\rangle_{\mathbb{H}}$.
\end{lemma}
\begin{proof}
	Proof follows in similar lines of  \cite[Lemma 3.1]{Feng}
\end{proof}
Let $\{e_{k}: k = 1, 2, 3, \cdots \}$ be an orthonormal basis in $H^{Jm}_{+}$. Then for every $x \in H^{Jm}_{+}$, we have
$ x = \sum\limits_{k=1}^{\infty} e_{k} \left\langle e_{k}| x \right\rangle$. Then we define the conjugate of $x$ as
\begin{equation*}
\overline{x} = \sum\limits_{k=1}^{\infty} e_{k} \overline{\left\langle e_{k}| x\right\rangle}.
\end{equation*}

If $A \in \mathcal{B}(H^{Jm}_{+})$,  then  $A(x) = \sum\limits_{k=1}^{\infty} e_{k} \left\langle e_{k}|Ax \right\rangle$. The conjugate of $A$ is defined by
\begin{equation*}
\overline{A}(x) = \sum\limits_{k=1}^{\infty} e_{k} \overline{\left\langle e_{k}|Ax \right\rangle}.
\end{equation*}

Moreover if $(q_{rs})$ is the matrix of $A$ with respect $\{e_{k}: k = 1,2,3, \cdots\}$, then   $(\overline{q}_{rs})$ and $(q_{rs} \cdot n)$ represents the matrices of $\overline{A}$ and $ A \cdot n$ respectively.

If $x \in H$ then $x = x_{1}+x_{2} \cdot n$, where $x_{1},x_{2} \in H^{Jm}_{+}$.
\begin{thm}\label{associate}
Let $A \in \mathcal{B}(H)$ has the matrix representation, $( q_{rs})  = (\alpha_{rs}) + (\beta_{rs}) \cdot n$ with an orthonormal basis $\{f_{k}: k = 1,2,3, \cdots \}$.  Then there exists a unique pair of operators $A_{1}, A_{2} \in \mathcal{B}(H^{Jm}_{+})$ such that
\begin{equation*}
A(x_{1}+x_{2}\cdot n) = (A_{1}x_{1}-A_{2}\overline{x}_{2}) +(A_{1}x_{2}+A_{2}\overline{x}_{2}), \; \text{for all}\; x_{1}
+x_{2} \cdot n \in H.
\end{equation*} Furthermore, $A_{1}, A_{2}$ admits matrix representations $(\alpha_{rs}), (\beta_{rs})$ with orthonoraml basis $\{f_{k}: k = 1,2,3, \cdots \}$,  respectively.
\end{thm}
\begin{proof}
	Using Lemma \ref{innerdirect}, the proof follows in the similar lines of \cite[Theorem 3.3]{Feng}.
\end{proof}
 Let $A \in \mathcal{B}(H).$ Then the adjoint of $A$ is denoted by $A^{*}$ and it is given by
\begin{equation*}
A^{*}(y_{1}+y_{2}\cdot n) = (A_{1}^{*}y_{1}+\overline{A}_{2}^{*}\overline{y}_{2}) + (A_{1}^{*}y_{2}-\overline{A}_{2}^{*}\overline{y_{1}})\cdot n, \; \text{for all} \;y_{1}+y_{2}\cdot n \in H.
\end{equation*}
\begin{prop}
Let $A = A_{1}+A_{2} \cdot n \in \mathcal{B}(H)$. Define  $\chi_{A} \colon H^{Jm}_{+} \oplus H^{Jm}_{+} \to H^{Jm}_{+} \oplus H^{Jm}_{+}$ by
\begin{align*}
\chi_{A}\begin{pmatrix} x_{1} \\ x_{2} \end{pmatrix} &= \begin{pmatrix} A_{1}x_{1}+ A_{2}x_{2} \\ -\overline{A}_{2}x_{2}+ \overline{A}_{1}x_{2} \end{pmatrix} \\
&= \begin{pmatrix} A_{1} & {A}_{2} \\
-\overline{A_{2}} & \overline{A_{1}}\end{pmatrix}\begin{pmatrix} x_{1} \\ x_{2} \end{pmatrix} ,\; \text{for all} \; \begin{pmatrix} x_{1} \\ x_{2} \end{pmatrix} \in H^{Jm}_{+} \oplus H^{Jm}_{+}.
\end{align*}
Then  $\chi_{A} \in \mathcal B \left(H^{Jm}_{+}\oplus H^{Jm}_{+}\right)$ and   $\|\chi_{A} \| = \|A\|$.
\end{prop}
\begin{proof}
Let $\begin{pmatrix} x_{1} \\ x_{2} \end{pmatrix}, \begin{pmatrix} y_{1} \\ y_{2} \end{pmatrix} \in H^{Jm}_{+}\oplus H^{Jm}_{+}$ and $\lambda \in \mathbb{C}_{m}.$ Then
\begin{align*}
\chi_{A} \left(\begin{pmatrix} x_{1} \\ x_{2} \end{pmatrix} + \lambda \begin{pmatrix} y_{1} \\ y_{2} \end{pmatrix}\right) &= \chi_{A}\left(\begin{pmatrix} x_{1}+\lambda y_{1} \\ x_{2}+\lambda y_{2} \end{pmatrix}\right) \\
                       &=\begin{pmatrix} A_{1}(x_{1} + \lambda y_{1}) + A_{2}(x_{2}+\lambda y_{2}) \\ -\overline{A}_{2}(x_{1}+\lambda y_{1}) + \overline{A}_{1}(x_{2}+\lambda y_{2}) \end{pmatrix} \\
											 &= \begin{pmatrix} A_{1}x_{1} +  A_{2}x_{2} \\ -\overline{A}_{2}x_{1}+ \overline{A}_{1}x_{2} \end{pmatrix} + \lambda \begin{pmatrix} A_{1} y_{1} + A_{2} y_{2} \\ -\overline{A}_{2}y_{1} + \overline{A}_{1}y_{2} \end{pmatrix} \\
											 & = \chi_{A} \begin{pmatrix} x_{1} \\ x_{2} \end{pmatrix} + \lambda \ \cdot \chi_{A} \begin{pmatrix} y_{1} \\ y_{2} \end{pmatrix}.
\end{align*}

This ensures that $\chi_{A}$ is $\mathbb{C}_{m}$ - linear operator on $H^{Jm}_{+} \oplus H^{Jm}_{+}.$ It is remains to show that $\chi_{A}$ is bounded and compute its norm. We have
\begin{align*}
\left\|\chi_{A}\begin{pmatrix} x_{1} \\ x_{2} \end{pmatrix}\right\|^{2} &= \left\|\begin{pmatrix} A_{1}x_{1} +  + A_{2}x_{2} \\ -\overline{A}_{2}x_{1}+ \overline{A}_{1}x_{2} \end{pmatrix}\right\|^{2} \\
&= \|A_{1}x_{1} + A_{2}x_{2}\|^{2} + \|-\overline{A}_{2}x_{1}+ \overline{A}_{1}x_{2}\|^{2} \\
&= \|A_{1}x_{1} + A_{2}x_{2}\|^{2} + \|\overline{-\overline{A}_{2}x_{1}+ \overline{A}_{1}x_{2}}\|^{2} \\
&= \|A_{1}x_{1} + A_{2}x_{2}\|^{2} + \|{A}_{2}\overline{x_{1}} - A_{1}\overline{x}_{2} \|^{2} \\
&= \|(A_{1}x_{1} + A_{2}x_{2}) + ({A}_{2}\overline{x_{1}} - A_{1}\overline{x}_{2}) \cdot n\|^{2}\\
&= \|(A_{1}+A_{2}\cdot n)(x_{1}-\overline{x}_{2} \cdot n)\| \\
&= \|A((x_{1}-\overline{x}_{2} \cdot n))\|^{2} \\
&\leq \|A\|^{2} (\|x_{1}\|^{2} + \|x_{2}\|^{2})\\
&= \|A\|^{2} \left\|\begin{pmatrix} x_{1} \\ x_{2} \end{pmatrix}\right\|^{2}.
\end{align*}
Therefore $\|\chi_{A}\| \leq \|A\|$.

If $\begin{pmatrix} x_{1} \\ x_{2} \end{pmatrix} \in H^{Jm}_{+} \oplus H^{Jm}_{+}$ with $\left\|\begin{pmatrix} x_{1} \\ x_{2} \end{pmatrix} \right\|^{2} = \|x_{1}\|^{2} + \|x_{2}\|^{2} = 1,$ then $x_{1} - \overline{x}_{2}\cdot n \in H$ with $\|x_{1}-\overline{x}_{2}\cdot n\|^{2}= 1$ and $\left\|\chi_{A}\begin{pmatrix} x_{1} \\ x_{2} \end{pmatrix}\right\| = \|A((x_{1}-\overline{x}_{2} \cdot n))\|.$ Therefore
\begin{equation*}
\sup_{x\in H, \|x\|=1} \|Ax\| = \sup_{y \in H^{Jm}_{+}\oplus H^{Jm}_{+}, \|y\|=1} \|\chi_{A}y\|.
\end{equation*}
Thus $\|A\| = \|\chi_{A}\|.$
\end{proof}

\begin{prop}\label{injective}
The map $\xi \colon \mathcal{B}(H) \to \mathcal{B}(H^{Jm}_{+} \oplus H^{Jm}_{+})$ defined by
\begin{equation*}
\xi(A) = \chi_{A} \colon= \begin{pmatrix} A_{1} & A_{2} \\ -\overline{A}_{2} & \overline{A}_{1}\end{pmatrix}, \; \text{for all}\; A \in \mathcal{B}(H),
\end{equation*}
is an injective $\mathbb{C}_{m}$- algebra homomorphism. Here $A_{1}, A_{2} \in \mathcal{B}(H^{Jm}_{+})$ are associated to $A$ as in Theorem \ref{associate}.
\end{prop}
\begin{proof} Let $A , B \in \mathcal{B}(H)$ and $\lambda \in \mathbb{C}_{m}.$ Then
\begin{align*}
\xi(A+\lambda B) &= \chi_{(A+ \lambda B)} \\
         &= \begin{pmatrix}A_{1}+\lambda B_{1} & A_{2}+\lambda B_{2} \\ -\overline{A}_{2}-\overline{\lambda} \overline{B}_{2} & \overline{A}_{1}+\overline{\lambda} \overline{B}_{1}\end{pmatrix}\\
				 &= \begin{pmatrix} A_{1} & A_{2} \\ -\overline{A}_{2} & \overline{A}_{1}\end{pmatrix} + \lambda \begin{pmatrix} B_{1} & B_{2} \\ -\overline{B}_{2} & \overline{B}_{1}\end{pmatrix}\\
				& = \chi_{A} + \lambda \cdot \chi_{B} \\
				& = \xi(A) + \lambda \cdot \xi(B).
\end{align*}
Next,
\begin{align*}
\xi(A\cdot B) &= \chi_{A\cdot B} \\
              &=  \begin{pmatrix}A_{1}B_{1} - A_{2} \overline{B}_{2} &  A_{1}B_{2}+A_{2}\overline{B}_{1} \\
							    -\overline{A}_{1}\overline{B}_{2}-\overline{A}_{2} {B}_{1} & \overline{A}_{1}\overline{B}_{1} - \overline{A}_{2} {B}_{2}\end{pmatrix}\\
							&= \begin{pmatrix} A_{1} & A_{2} \\ -\overline{A}_{2} & \overline{A}_{1}\end{pmatrix} \cdot \begin{pmatrix} B_{1} & B_{2} \\ -\overline{B}_{2} & \overline{B}_{1}\end{pmatrix} 	\\
							&= \chi_{A} \cdot \chi_{B} \\
							&= \xi(A) \cdot \xi(B).
\end{align*}
If $\xi(A) = 0,$ then $\chi_{A} = \begin{pmatrix} A_{1} & A_{2} \\ -\overline{A}_{2} & \overline{A}_{1}\end{pmatrix} = \begin{pmatrix} 0 & 0 \\ 0 & 0\end{pmatrix}.$  Therefore $A = 0.$ Also using the fact that $\xi (0) = 0,$ we can conclude that $\xi$ is injective complex algebra homomorphism.\qedhere
\end{proof}

\begin{thm} \label{equivalents}
 Let $A \in \mathcal{B}(H)$. Then
\end{thm}
\begin{enumerate}
\item $\chi_{A^{*}} = \chi_{A}^{*}$
\item $A$ is self-adjoint $\Leftrightarrow \chi_{A}$ is self-adjoint
\item $A$ is positive $\Leftrightarrow \chi_{A}$ is positive
\item $A$ is normal $\Leftrightarrow \chi_{A}$ is normal
\item $A$ is unitary $\Leftrightarrow \chi_{A}$ is unitary.
\item $A$ is anti self-adjoint $\Leftrightarrow \chi_{A}$ is anti self-adjoint.
\end{enumerate}
\begin{proof}
Proof of $(1):$ Let $A = (a_{rs}) + (b_{rs}) \cdot n$ be the matrix representation of $A,$ where $A_{1}=  (a_{rs})$ and $A_{2}= (b_{rs}).$ Then
\begin{equation*}
A^{*}= (\overline{a_{sr}}) - (b_{sr}) \cdot n.
\end{equation*}
and
\begin{equation*}
\chi_{A^{*}}= \begin{pmatrix} (\overline{a}_{sr}) & - (b_{sr}) \\ (\overline{b}_{sr}) & (a_{sr})\end{pmatrix}.
\end{equation*}
We have
\begin{align*}
\big(\chi_{A}\big)^{*} & = \begin{pmatrix} (a_{rs}) & (b_{rs}) \\ -(\overline{b}_{rs}) & (\overline{a}_{rs})\end{pmatrix}^{*}\\
&= \begin{pmatrix} (a_{rs})^{*} &  -(\overline{b}_{rs})^{*}\\  (b_{rs})^{*} & (\overline{a}_{rs})^{*}\end{pmatrix}\\
&= \begin{pmatrix} (\overline{a}_{sr}) & - (b_{sr}) \\ (\overline{b}_{sr}) & (a_{sr})\end{pmatrix}\\
&= \chi_{A^{*}}.
\end{align*}
Proof of $(2):$ From Proposition \ref{injective} and $(1),$ we have

\begin{equation*}
A = A^{*} \Leftrightarrow \chi_{A} = \chi_{A^{*}}\Leftrightarrow \chi_{A} = \chi_{A}^{*}\Leftrightarrow \chi_{A}.
\end{equation*}
Proof of $(3):$ For $x_{1}, x_{2} \in H^{Jm}_{+}$, we have
\begin{align*}
A \; \text{is positive} &\Leftrightarrow A = A^{*}, \left\langle x_{1}+x_{2}\cdot n | A(x_{1}+x_{2}\cdot n)\right\rangle \geq 0  \\
& \Leftrightarrow \chi_{A} = \chi_{A}^{*}, \left\langle \begin{pmatrix} x_{1} \\ -\overline{x}_{2} \end{pmatrix}| \chi_{A}\begin{pmatrix} x_{1} \\ -\overline{x}_{2} \end{pmatrix} \right\rangle \geq 0.\\
&\Leftrightarrow \chi_{A}\; \text{is positive}.
\end{align*}
Proof of $(4):$
\begin{equation*}
 AA^{*}=A^{*}A \Leftrightarrow \chi_{AA^{*}} = \chi_{A^{*}A}
										\Leftrightarrow \chi_{A} \big(\chi_{A}\big)^{*} = \big(\chi_{A}\big)^{*}\chi_{A}.					
\end{equation*}
Proof of $(5):$
\begin{align*}
A^{*}A = AA^{*}=I_{H} &\Leftrightarrow \chi_{A^{*}A} = \chi_{AA^{*}} = \chi_{I_{H}}\\ &\Leftrightarrow\chi_{A}^{*} \chi_{A} = \chi_{A} \chi_{A}^{*} = I_{H^{Jm}_{+} \oplus H^{Jm}_{+}}.
\end{align*}
Proof of $(6):$
\begin{equation*}
A^{*}= - A \Leftrightarrow \chi_{{A}^{*}} = \chi_{-A} \Leftrightarrow \chi_{A}^{*} = - \chi_{A}.  \qedhere
\end{equation*}

\end{proof}
\begin{lemma}\label{comparision}
	Let $A \in \mathcal{B}(H)$. Then,
\end{lemma}
\begin{enumerate}
\item $ x_{1} + x_{2} \cdot n \in N(A) \Leftrightarrow \begin{pmatrix} x_{1} \\ -\overline{x}_{2} \end{pmatrix}\in {N}(\chi_{A})$ {and} $ x_{1} + x_{2} \cdot n \in R(A) \Leftrightarrow \begin{pmatrix} x_{1} \\ -\overline{x}_{2} \end{pmatrix}\in R(\chi_{A}).$
\item $ y_{1}+ y_{2} \cdot n \in N(A)^{\bot} \Leftrightarrow \begin{pmatrix} y_{1} \\ -\overline{y}_{2} \end{pmatrix} \in N(\chi_{A})^{\bot}$ and $y_{1}+ y_{2} \cdot n \in R(A)^{\bot} \Leftrightarrow \begin{pmatrix} y_{1} \\ -\overline{y}_{2} \end{pmatrix} \in R(\chi_{A})^{\bot}.$
\end{enumerate}
\begin{proof}
Proof of $(1):$
Let $x_{1} + x_{2} \cdot n \in H.$ Then
\begin{align*}
x_{1} + x_{2} \cdot n \in N(A) &\Leftrightarrow A(x_{1} + x_{2} \cdot n) = 0 \\
                                      &\Leftrightarrow (A_{1}x_{1} - A_{2}\overline{x}_{2}) + (A_{1}x_{2}+A_{2}\overline{x}_{1}) \cdot n = 0 \\
																			& \Leftrightarrow A_{1}x_{1} - A_{2}\overline{x}_{2} = 0 \; \text{and} \;  A_{1}x_{2}+A_{2}\overline{x}_{1} = 0 \\
																			& \Leftrightarrow A_{1}x_{1} - A_{2}\overline{x}_{2} = 0 \; \text{and} \; -\overline{A}_{1}\overline{x}_{1}-\overline{A}_{2}x_{1} = 0 \\
																			& \Leftrightarrow \chi_{A}\begin{pmatrix} x_{1} \\ -\overline{x}_{2}\end{pmatrix} = 0 \\
																			& \Leftrightarrow \begin{pmatrix} x_{1} \\ -\overline{x}_{2}\end{pmatrix} \in N(\chi_{A}).
\end{align*}
If $x_{1}+ x_{2} \cdot n \in {R}(A), $ then $x_{1}+x_{2} \cdot n = A(w_{1}+w_{2}\cdot n),$ for some $w_{1} + w_{2} \cdot n \in H.$ We have
\begin{align*}
     x_{1}+ x_{2} \cdot n \in R(A)  & \Leftrightarrow x_{1}+x_{2} \cdot n = A(w_{1}+w_{2}\cdot n)\\
                                            &\Leftrightarrow A_{1}w_{1}-A_{2}\overline{w}_{2} = x_{1} \; \text{and} \; \ A_{1}w_{2}+A_{2}\overline{w}_{1} = x_{2} \\
                                            &\Leftrightarrow A_{1}w_{1}-A_{2}\overline{w}_{2} = x_{1} \; \text{and} \; -\overline{A}_{1}\overline{w}_{2} - \overline{A}_{2}\overline{w}_{1} = -\overline{x}_{2} \\
																						&\Leftrightarrow \begin{pmatrix} A_{1} & A_{2} \\ -\overline{A}_{2} & \overline{A}_{1}\end{pmatrix}\begin{pmatrix} w_{1} \\ -\overline{w}_{2}\end{pmatrix} = \begin{pmatrix} x_{1} \\ -\overline{x}_{2}\end{pmatrix} \\
																						&\Leftrightarrow \chi_{A}\begin{pmatrix} w_{1} \\ -\overline{w}_{2}\end{pmatrix} = \begin{pmatrix} x_{1} \\ -\overline{x}_{2}\end{pmatrix} \\
																						&\Leftrightarrow \begin{pmatrix} x_{1} \\ -\overline{x}_{2}\end{pmatrix} \in R(\chi_{A}).
\end{align*}
Proof of $(2):$ Let $y_{1}+y_{2}\cdot n \in H.$ Then
\begin{equation*}
y_{1} + y_{2} \cdot n \in N(A)^{\bot}= \overline{R(A^{*})} \Leftrightarrow \begin{pmatrix} y_{1} \\ -\overline{y}_{2}\end{pmatrix} \in \overline{R(\chi_{A}^{*})} = N(\chi_{A})^{\bot}.
\end{equation*}
Similarly,
\begin{align*}
y_{1} + y_{2} \cdot n \in {R}(A)^{\bot} = N(A^{*}) & \Leftrightarrow \begin{pmatrix} y_{1} \\ -\overline{y}_{2}\end{pmatrix} \in {N}(\chi_{A^{*}}) = {R}(\chi_{A})^{\bot}. \qedhere
\end{align*}
\end{proof}
\begin{lemma}\label{iso}
	 Let $A \in \mathcal{B}(H)$. Then
	\begin{enumerate}
		\item $A$ is orthogonal projection if and only if $\chi_{A}$ is orthogonal projection
		\item $A$ is partial isometry if and only if $\chi_{A}$ is partial isometry.
	\end{enumerate}
 \end{lemma}
\begin{proof} Proof of $(1):$ By Theorem \ref{comparision} and Proposition \ref{equivalents},
	
\begin{align*}
	A^{2}=A \; \text{and} \; A^{*}=A &\Leftrightarrow \chi_{A^{2}} = \chi_{A} \; \text{and} \; (\chi_{A})^{*} = \chi_{A} \\
	& \Leftrightarrow (\chi_{A})^{2}= \chi_{A} \; \text{and} \; (\chi_{A})^{*} = \chi_{A}.
\end{align*}
\noindent	
Proof of $(2):$
If $A$ is a partial isometry, then
\begin{equation*}
\|A(x_{1}+x_{2}\cdot n)\| = \|x_{1} + x_{2} \cdot n\|, \; \text{for all}\; x_{1}+x_{2}\cdot n \in N(A)^{\bot}.
\end{equation*}
By Lemma \ref{comparision}(2), for every $\begin{pmatrix} x_{1} \\ {x}_{2}\end{pmatrix} \in N(\chi_{V})^{\bot},$ we have $x_{1}-\overline{x}_{2} \cdot n \in N(V)^{\bot}.$ Therefore
\begin{align*}
\left\|\chi_{A} \begin{pmatrix} x_{1} \\ {x}_{2}\end{pmatrix}\right\|^{2} &= \left\|\begin{pmatrix} A_{1} & A_{2} \\ -\overline{A}_{2} & \overline{A}_{1}\end{pmatrix}\begin{pmatrix} x_{1} \\ {x}_{2}\end{pmatrix}\right\|^{2} \\
 & = \left\|\begin{pmatrix} A_{1}x_{1}+A_{2}{x}_{2} \\ -\overline{A}_{2}x_{1}+ \overline{A}_{1}{x}_{2}\end{pmatrix}\right\|^{2} \\
 & = \|A_{1}x_{1}+A_{2}{x}_{2}\|^{2} + \|-\overline{A}_{2}x_{1}+ \overline{A}_{1}{x}_{2}\|^{2} \\
 & = \|A(x_{1}+{x}_{2}\cdot n)\|^{2}\\
 & = \|x_{1}+{x}_{2}\cdot n\|^{2}\\
 & =  \left\|\begin{pmatrix} x_{1} \\ {x}_{2}\end{pmatrix}\right\|^{2}.
\end{align*}
We conclude that $A$ is partial isometry if and only if $\chi_{A}$ is partial isometry.
\end{proof}
\begin{prop}\label{frompolar}
Let $A \in \mathcal{B}(H)$. Suppose $U$ be a partial isometry such that $A = U|A|$ and $A = U_{0}|A|$ be the polar decomposition of $A$. Then $U = U_{0} + VP$ for some partial isometry $V$ and $P = P_{N(A)}$.
\end{prop}
\begin{proof}
	Proof follows in the similar lines of \cite[Proposition 2.5]{ichinose}
\end{proof}
\begin{thm}\label{main}
Let $A\in \mathcal{B}(H) $ and $U$ be a partial isometry with $A = U|A|$.  Then  $U=U_{0}$ if and only if either $N(A) = \{0\}$ or $R(A)^{\bot} = \{0\}.$
\end{thm}
\begin{proof} Since $A = U|A|$, we have  $\chi_{A} = \chi_{(U|A|)} = \chi_{U} \cdot \chi_{|A|}$.
By Theorem \ref{iso}(2), it is clear that $\chi_{U}$ is partial isometry. Since $|A|^{2}= A^{*}A,$ we have
\begin{equation*}
|\chi_{A}|^{2} = \chi_{A}^{*}\chi_{A} = \chi_{A^{*}} \chi_{A}= \chi_{A^{*}A}= \chi_{|A|^{2}} = \chi_{|A|}^{2}.
\end{equation*}
 It is clear from \cite[Theorem 2]{ichinose}, that $\chi_{U} = \chi_{U_{0}}$  if and only if either $N(\chi_{A}) = \{0\}$ or $R(\chi_{A})^{\bot} = \{0\},$ but Lemma \ref{comparision} ensures that $N(\chi_{A}) = \{0\}$ or $R(\chi_{A})^{\bot} = \{0\}$ if and only if  $N({A}) = \{0\}$ or $R({A})^{\bot} = \{0\}.$
Thus $U = U_{0}$ if and only if either $N(A) = \{0\}$ or $R(A)^{\bot} = \{0\}.$
\end{proof}
We give an example of a bounded right $\mathbb{H}$- linear operator $A$ with $N(A) \neq \{0\}$ and $N(A^{*})\neq \{0\}$ implies there is a partial isometry $U$ such that $A = U|A|$ with $U \neq U_{0}$.
\begin{eg}\label{Example}Define $A \colon \ell^{2}(\mathbb{N}, \mathbb{H}) \to \ell^{2}(\mathbb{N}, \mathbb{H}) $ by
	\begin{equation*}
	A\big(x_{1}, x_{2}, x_{3}, \cdots\big) = \big(0, \frac{x_{1}}{\sqrt{2}},0,\frac{x_{2}}{\sqrt{2}},0, 6\frac{x_{6}}{\sqrt{37}},7\frac{x_{7}}{\sqrt{50}} \cdots \big), \; \text{for all} \; (x_{n}) \in \ell^{2}(\mathbb{N}, \mathbb{H}).
	\end{equation*}
	Then $A$ is bounded right $\mathbb{H}$- linear operator. The adjoint $A^{*}$  is given by
	\begin{equation*}
	A^{*}\big(x_{1}, x_{2}, x_{3}, \cdots \big) = \big(\frac{x_{2}}{\sqrt{2}}, \frac{x_{4}}{\sqrt{2}}, 0, 0 , 0, 6\frac{x_{6}}{\sqrt{37}}, 7\frac{x_{7}}{\sqrt{50}} \cdots\big), \; \text{for all} \; (x_{n}) \in \ell^{2}(\mathbb{N}, \mathbb{H}).
	\end{equation*}
	Here $N(A) = span\{e_{3}, e_{4}, e_{5}\}$ and $N(A^{*}) = span \{e_{1}, e_{3}, e_{5}\}$. The modulus of $A$ is given by
	 \begin{equation*}
	 |A|\big(x_{1}, x_{2}, x_{3}, \cdots\big) = \big(\frac{x_{1}}{\sqrt{2}}, \frac{x_{2}}{\sqrt{2}},0,0,0, 6\frac{x_{6}}{\sqrt{37}},7\frac{x_{7}}{\sqrt{50}}, \cdots \big).
	 \end{equation*}
	
	  Define $U_{0} \colon \ell^{2}(\mathbb{N}, \mathbb{H}) \to \ell^{2}(\mathbb{N}, \mathbb{H}) $ by
	  \begin{equation*}
	  U_{0}\big(x_{1}, x_{2}, x_{3}, \cdots\big) = \big(x_{1},x_{2},0,0,0,x_{6},x_{7}, \cdots \big),\; \text{for all} \; (x_{n}) \in \ell^{2}(\mathbb{N}, \mathbb{H}).
	  \end{equation*}
	  Then $U_{0}$ is a partial isometry with $N(U_{0}) = N(A)$ and $A = U_{0} |A|$. Also define
	  \begin{equation*}
	  U = U_{0} + VP,
	  \end{equation*}
	   where $P$ is an orthogonal projection onto $N(A)$. This implies $P|A| = 0$. Define $V \colon N(A) \to R(A)^{\bot}$ by
	   \begin{equation*}
	   V(e_{3} \cdot \alpha + e_{4} \cdot \beta + e_{5} \cdot \gamma) = e_{1} \cdot \alpha + e_{3} \cdot \beta + e_{5} \cdot \gamma, \; \text{for every} \; \alpha, \beta, \gamma \in \mathbb{H}.
	   \end{equation*}
	    Let $U = U_{0} + VP$. We show that $A=U|A|$ and $U \neq U_{0}$. It is clear that $U|A| = U_{0}|A| + VP|A| = U_{0}|A| = A $. We have $U_{0}(e_{4}) = 0$ but  $U(e_{4}) = U_{0}(e_{4})+VP(e_{4}) = e_{3} \neq 0$. Hence $U \neq U_{0}$.
\end{eg}
 The Proposition \ref{frompolar} is also true for densely defined closed right linear operators in $H$. We examine this with the following example.
\begin{eg}
	Define $S \colon \mathcal{D}(S) \to \ell^{2}(\mathbb{N}, \mathbb{H})$ by
	\begin{equation*}
	S\big(x_{1}, x_{2}, x_{3}, \cdots \big) = \big(0, x_{1}, 0, x_{2}, 0 , 6x_{6}, 7x_{7}, \cdots\big), \; \text{for all} \;(x_{n}) \in \mathcal{D}(S),
	\end{equation*}
	where
	\begin{equation*}
	\mathcal{D}(S) = \left\{(x_{1}, x_{2}, x_{3}, \cdots) \in \ell^{2}(\mathbb{N}, \mathbb{H}) \colon (0, x_{1}, 0, x_{2}, 0 , 6x_{6}, 7x_{7}, \cdots ) \in \ell^{2}(\mathbb{N}, \mathbb{H}) \right\}.
	\end{equation*}
	It is clear that $S$ is right $\mathbb{H}$- linear unbounded operator. We show that $S$ is densely defined closed operator. Since $C_{00}$ is a dense subspace of $\ell^{2}(\mathbb{N}, \mathbb{H})$ and $C_{00} \subseteq \mathcal{D}(S) $, we conclude that $\mathcal{D}(S)$ is dense in $\ell^{2}(\mathbb{N}, \mathbb{H})$. Hence $S$ is densely defined.
	Now we show that $S$ is closed operator.
	
	Let $(z_{n})$ be sequence in $\mathcal{D}(S)$, where $z_{n} = (z_{n}^{(j)})_{j \in \mathbb{N}}$ such that $(z_{n})$ converges $z = (z^{(1)}, z^{(2)}, z^{(3)}, \cdots)$ and $Sz_{n}$ converges to $y \in \ell^{2}(\mathbb{N}, \mathbb{H})$. Then for each fixed $j$ the column  sequence $(z_{n}^{(j)})$ converges to  $ z^{(j)} $ as $n \to \infty$ and by the definition of $S$, we have $Sz_{n}$ converges to $Sz$. Therefore $Sz = y$. The adjoint of $S$ is given by
	\begin{equation*}
	S^{*}\big(x_{1}, x_{2}, x_{3}, \cdots \big) = \big(x_{2},x_{4}, 0 ,0,0,  6x_{6}, 7x_{7}, \cdots\big), \; \text{for all} \;(x_{n}) \in \mathcal{D}(S^{*}).
	\end{equation*}
	Here $N(S) \neq \{0\}$ and $R(S)^{\bot} = N(S^{*}) \neq \{0\}$. The bounded transform of $S$ is given by $\mathcal{Z}_{S} = S(I+S^{*}S)^{-\frac{1}{2}}$, that is
	\begin{equation*}
	\mathcal{Z}_{S}\big(x_{1}, x_{2},x_{3}, \cdots \big) = \big(0, \frac{x_{1}}{\sqrt{2}},0, \frac{x_{2}}{\sqrt{2}},0, \frac{6x_{6}}{\sqrt{37}}, \frac{7x_{7}}{\sqrt{50}}, \cdots\big), \; \text{for all} \;(x_{n}) \in \ell^{2}(\mathbb{N}, \mathbb{H}).
	\end{equation*}
	This co-insides with the bounded operator given in Example \ref{Example}.  By Theorem \ref{unboundedunique}, for any partial isometry $U$ with $\mathcal{Z_{S}} = U|\mathcal{Z}_{S}|$, we have $S = U|S|$. In fact $S = U_{0}|S|$, where $U_{0}$ is same as given in Example \ref{Example}.
	
\end{eg}
Now we prove the uniqueness  theorem by using the bounded transform.

\begin{thm}\label{unboundedunique}
	Let $T\colon \mathcal{D}(T)\to H$ be closed and densely defined right $\mathbb{H}-$ linear. Let $U$ be a partial isometry such that $T = U|T|$. Then $U= U_{0}$ if and only if either $N(U)=\{0\}$ or $R(U)^{\bot}=\{0\}.$
\end{thm}
\begin{proof} By Theorem \ref{boundedtransform}, we have
	\begin{equation} \label{|Z_{T}|equation}
	Z_{T} = T(I+T^{*}T)^{-\frac{1}{2}} = U|T|(I+T^{*}T)^{-\frac{1}{2}},
	\end{equation}
	a bounded operator with
	\begin{align*}
	|Z_{T}|^{2} = Z_{T}^{*}Z_{T} &= (I+T^{*}T)^{-\frac{1}{2}}T^{*}T(I+T^{*}T)^{-\frac{1}{2}} \\
	&= T^{*}T(I+T^{*}T)^{-1} \\
	&= (|T|(I+T^{*}T)^{-\frac{1}{2}})^{2}.
	\end{align*}
	Here we have used the fact that since $T^{*}T$ commute with $I+T^{*}T$, it commute with $(I+T^{*}T)^{-\frac{1}{2}}$.
	First, note that $|T|(I+T^{*}T)^{-\frac{1}{2}}$ is positive. By the uniqueness of the square root $|Z_{T}|= |T|(I+T^{*}T)^{-\frac{1}{2}}$, we have $Z_{T} = U|Z_{T}|.$ Applying Theorem \ref{main} to the bounded right $\mathbb{H}-$ linear operator $Z_{T},$ we obtain $U=U_{0}$  if and only if either $N(Z_{T})= \{0\}$ or $R(Z_{T}) = \{0\}.$ By Lemma \ref{null}, $U=U_{0}$ if and only if either $N(T)=\{0\}$ or $R(T)^{\bot}=\{0\}.$
\end{proof}
\section*{Acknowledgment}
The second author is thankful to INSPIRE (DST) for the support in the form of fellowship (No. DST/ INSPIRE Fellowship/2012/IF120551), Govt of India.

\end{document}